\documentclass[12pt, cd]{amsart} 
\usepackage{graphicx,amssymb}
\usepackage{mathrsfs}
\usepackage{amssymb,amsmath,amsthm,color}
\usepackage{graphicx}
\usepackage{hyperref}
\usepackage{url} 
\usepackage{setspace}
\usepackage[margin=1in]{geometry}

\numberwithin{equation}{section}

\newcommand{\R}{\mathbb R}
\newcommand{\C}{\mathbb C}

\def\TagOnRight

\def\R {\mathbb{R}}

\newcommand{\be}{\begin{equation}}
\newcommand{\ee}{\end{equation}}
\newcommand{\bea}{\begin{eqnarray}}
\newcommand{\eea}{\end{eqnarray}}
\newcommand{\Bea}{\begin{eqnarray*}}
\newcommand{\Eea}{\end{eqnarray*}}
\newcommand{\bt}{\begin{Theorem}}
\newcommand{\et}{\end{Theorem}}
\newcommand{\bpr}{\begin{Proposition}}
\newcommand{\epr}{\end{Proposition}}

\newcommand{\bl}{\begin{Lemma}}
\newcommand{\el}{\end{Lemma}}
\newcommand{\bi}{\begin{itemize}}
\newcommand{\ei}{\end{itemize}}

\newtheorem{Definition}{Definition}[section]
\newtheorem{Theorem}[Definition]{Theorem}
\newtheorem{Lemma}[Definition]{Lemma}
\newtheorem{Proposition}[Definition]{Proposition}

\newtheorem{Remark}[Definition]{Remark}

\usepackage{mathrsfs}
\usepackage{amsthm}
\usepackage{hyperref}
\usepackage{comment}

\begin{document}
\baselineskip16pt

\title[The NlS with harmonic potential in modulation spaces ]{The nonlinear Schr\"odinger  equations with Harmonic Potential in modulation spaces}
\author{Divyang G. Bhimani}
\address{Department of Mathematics\\
University of Maryland\\
College Park\\ 
MD 20742}
\email{dbhimani@math.umd.edu}

\maketitle
\begin{abstract} 
We study nonlinear Schr\"odinger $i\partial_tu-Hu=F(u)$   (NLSH)  equation associated to  harmonic oscillator $H=-\Delta +|x|^2$ in modulation spaces $M^{p,q}.$  When $F(u)= (|x|^{-\gamma}\ast |u|^2)u, $
we prove global well-posedness   for (NLSH)  in  modulation spaces $M^{p,p}(\mathbb R^d)$ 
$ (1\leq p  <  2d/(d+\gamma), 0<\gamma< \min \{ 2, d/2\}).$  When $F(u)= (K\ast |u|^{2k})u$ with $K\in \mathcal{F}L^q $ (Fourier-Lebesgue spaces) or $M^{\infty,1}$ (Sj\"ostrand's class) or $M^{1, \infty},$ some  local and global well-posedness  for (NLSH)   are obtained in some   modulation spaces.
When $F$ is real entire and $F(0)=0$, we prove local well-posedness  for (NLSH) in $M^{1,1}.$   As a consequence, we can get  local and global well-posedness  for (NLSH)   in a function spaces$-$which  are larger than usual $L^p_s-$Sobolev spaces.
\end{abstract}

\section{Introduction}
We study Cauchy problem for the  nonlinear Schr\"odinger  equation with the harmonic  oscillator $H=-\Delta +|x|^2$:
\begin{eqnarray}\label{HTEH}
i\partial_{t}u(t,x) -Hu(t,x) = F(u), \   u(x,0)= u_0(x),
\end{eqnarray}
where $u:\mathbb R_t \times \mathbb R_x^d \to \mathbb C, u_0:\mathbb R^d \to \mathbb C, \Delta = \sum_{1}^{d} \frac{\partial^2}{\partial^2_{x_i}}$ ,  and $F:\mathbb C \to \mathbb C$ is a nonlinearity.  

Mainly we  consider nonlinearity of the Hartree  and power type. Specifically, we study \eqref{HTEH}  with the  Hartree type nonlinearity
\begin{eqnarray}\label{htn}
F(u)=(K\ast |u|^{2k}) u,
\end{eqnarray}
where  $\ast$ denotes the convolution in $\mathbb R^d, k \in \mathbb N,$ and  
$K$ is of the following type:
\begin{eqnarray}\label{hk}
\   \ \ \ \ \ \ \ \ \ \ \ \  \ \ \ \   \   \  \   \ K(x)= \frac{\lambda}{|x|^{\gamma}}, (\lambda \in \mathbb R, \gamma >0, x\in \mathbb R^{d}),
\end{eqnarray}
\begin{eqnarray}\label{dk}
\ \ \ \ \ K\in \mathcal{F}L^q(\mathbb R^d) \ (1<q< \infty),
\end{eqnarray}
\begin{eqnarray}\label{sc}
K\in M^{\infty,1}(\mathbb R^d),
\end{eqnarray}
\begin{eqnarray}\label{di}
K\in M^{1,\infty}(\mathbb R^d),
\end{eqnarray}
where $\mathcal{F}L^{q}$ is a Fourier-Lebesgue space, and   $M^{\infty,1}(\mathbb R^d)  \supset \mathcal{F}L^1(\mathbb R^d)$ and $L^1(\mathbb R^d) \subset M^{1, \infty}(\mathbb R^d)$ are modulation spaces (see Definition \ref{ms} below). The homogeneous kernel of the form \eqref{hk} is known as Hartree potential. The kernel of the form \eqref{sc} is sometimes called Sj\"ostrand class (particular modulation space).  
We also study \eqref{HTEH}  when $F$ is  real entire and $F(0)=0$ (see Definition \ref{red} below).
 In this case   power-type  nonliearity  
\[F(u)= \pm |u|^{2k}u \ (k\in \mathbb N)\]
is covered, and in particular, when $F(u)=-|u|^2u,$ equation \eqref{HTEH}
is  the well-known 
Gross-Pitaevskii equation. 

The harmonic oscillator (also known as  Hermite operator) $H$ is a fundamental operator in quantum physics and in  analysis \cite{ST}.  Equation \eqref{HTEH} models Bose-Einstein condensates with attractive interparticle  interactions under a magnetic trap \cite{bsh, ts, jz, dgp}. The isotropic harmonic potential $|x|^2$ describes a magnetic field whose role is to confine the movement of particles \cite{bsh, ts, tw}.  A class of NLS with a ``nonlocal'' nonlinearity that we call ``Hartree type" occurs in the modeling of quantum semiconductor devices (see \cite{rc2} and the references therein, cf. \cite{pc}).

It is known that the free Schr\"odinger propagator $e^{it\Delta}$  is bounded on Lebesgue spaces $L^p(\mathbb R^d)$ if and only if $p=2.$  Hence, we cannot expect to solve Schr\"odinger equation on $L^p-$spaces.  This has inspired to study in other  function spaces  (e.g., modulation spaces $M^{p,q}$) arising in harmonic analysis.  In fact, in contrast to $L^p-$spaces,  in  \cite{bao, benyi}   it is  proved that the Schr\"odinger propagator $e^{it(-\Delta)^{\frac{\alpha}{2}}} \ (0 \leq \alpha \leq 2)$ is  bounded  on $M^{p,q}(\mathbb R^d)$ for all $1\leq p, q \leq \infty.$   Local well-posedness  results of the corresponding nonlinear equations, with nonlinearity of power-type, or more generic real entire, were obtained  in \cite{ambenyi, dgb-pkr}.  And global well-posedness results, with nonlinearity of power or Hartree type, were obtained in \cite{wang2, dgb, bao}. We refer to excellent survey \cite{rsw} and the reference therein for details.

Coming back to harmonic oscillator  $H,$ we note  that  well-posedness for \eqref{HTEH} with  Hartree and  power type nonlinearity were obtained in  the energy space in \cite{rc, rc1, rc2, rc, rc3, pc,  jz, ts}.  Recently Bhimani-Balhara-Thangavelu    \cite{drt} have proved that   Schr\"odinger  propagator associated with the harmonic potential  (see \eqref{dhm}  and  Theorems  \ref{mso} below) is bounded on $M^{p,p}(\mathbb R^d) \ (1\leq p < \infty)$ (cf. \cite{kks, ec}).  However,  there is no work so far for  nonlinear Schr\"odinger  equation with  harmonic potential in modulation spaces.  We also  note that Cauchy data in modulation spaces are rougher (see Proposition \ref{exa} below) than  $L^p-$Sobolev spaces (see definition in Subsection \ref{ns} below).  Taking these  considerations into account, we are inspired study \eqref{HTEH}  in modulation spaces.  Specifically, we have following theorem.

\begin{Theorem}\label{mt}  Let    $F(u)$   and  $K$ be given by  \eqref{htn} and \eqref{hk}  respectively with $k=1,$ and 
 $ 0<\gamma < \text{min} \{2, d/2\}, d\in \mathbb N$.
Assume that $u_{0}\in M^{p,p}(\mathbb R^{d}) $  where $ 1\leq p < \frac{2d}{d+ \gamma}.$Then
 there exists a  unique global solution of \eqref{HTEH} such that $$u\in C([0,\infty), M^{p,p}(\mathbb R^{d})) \cap L^{8/\gamma}_{loc}([0,\infty), L^{4d/(2d-\gamma)} (\mathbb R^d)) .$$
\end{Theorem}

We note that   up to now cannot know  \eqref{HTEH} is globally well-posed in   $L^{p}(\mathbb R^d)$ but  it is in  $M^{p,p}(\mathbb R^d) \subsetneq L^p(\mathbb R^d)$ for $1\leq p <2.$
Since $L^p-$Sobolev spaces $L^p_s(\mathbb R^d)\subsetneq  M^{p,p}(\mathbb R^d) $ for $s\geq d(\frac{2}{p}-1),1<p\leq 2$ and $L^1_s(\mathbb R^d)\subsetneq  M^{1,1}(\mathbb R^d) $ for $s> d$ (see  Proposition \ref{exa} below).  Theorem \ref{mt} reveals that  we can get the global well-posedness for \eqref{HTEH}  with Cauchy data  rougher than $L^p-$Sobolev spaces.  We do not know whether the range of $p$ in Theorem \ref{mt} is sharp or not for the Hartree potential.  However, if we take  potential from Sj\"ostrand's class, the range of $p$ can be improved.  Since modulation spaces  enlarges as their exponents are increasing (see Lemma \ref{rl} \eqref{ir} below), we  can solve   \eqref{HTEH} with Cauchy data in relatively more low regularity spaces.  Specifically, we  have  following theorem.
\begin{Theorem}\label{mtv} Let    $F(u)$  be given by  \eqref{htn} with $k=1.$
 Assume that  $K$ is given by \eqref{sc}
  and $u_{0}\in M^{p,p}(\mathbb R^{d}) \ (1\leq p \leq 2).$  Then there exists a  unique global solution of \eqref{HTEH} such that $$u\in C([0,\infty), M^{p,p}(\mathbb R^{d})).$$
\end{Theorem}
 Now we note  that formally the  solution of  \eqref{HTEH} satisfies (see for e.g., \cite{rc2, jz, pc}) the conservation of mass:
$$\|u(t)\|_{L^{2}}=\|u_0\|_{L^{2}}  \ \ (t\in\mathbb R^+),$$
and exploiting this mass conservation law,  below Theorem \ref{mso},   and techniques from  time-frequency analysis,  we prove global existence results (above Theorems \ref{mt} and \ref{mtv})  for  \eqref{HTEH}  with Hartree type nonlineariy.  For power type nonlinearity, it remains open problem to get the global existence (cf. \cite[p. 280]{rsw}), however, we can prove local existence  for \eqref{HTEH}  (see Theorem \ref{mtp} below). We also prove local well-posedness results  for \eqref{HTEH}  for Hartree type nonlineariy (see Theorem \ref{wht1} below). Specifically, we have  following theorems.
\begin{Theorem} \label{mtp} Let $F$ is real entire and $F(0)=0.$
 Assume that $u_{0}\in M^{1,1}(\R^d) .$     Then there exists $T^{*}=T^{\ast}(\|u_{0}\|_{M^{1,1}})$ such that \eqref{HTEH} has a unique solution $u\in C([0, T^{*}), M^{1,1} (\mathbb R^d)).$  
\end{Theorem}

\begin{Theorem}\label{wht1}
Let    $F(u)$  be given by  \eqref{htn}.
\begin{enumerate}
\item Assume that $K$ is given by \eqref{dk}  and
   $u_{0}\in M^{1,1}(\R^d).$   Then there exists $T^{*}=T^{\ast}(\|u_{0}\|_{M^{1,1} })$ such that \eqref{HTEH} has a unique solution $u\in C([0, T^{*}), M^{1,1} (\mathbb R^d)).$ 
   
\item Assume that $K$ is given by \eqref{dk}   with $ 1<q<2, k=1$ and  $u_{0}\in M^{\frac{2r}{2r-1},\frac{2r}{2r-1}}(\mathbb R^{d}) \  (q<r< \infty).$   Then there exists $T^{*}=T^{\ast}(\|u_{0}\|_{ M^{\frac{2r}{2r-1},\frac{2r}{2r-1}}})$ such that \eqref{HTEH} has a unique solution $u\in C([0, T^{*}), M^{\frac{2r}{2r-1},\frac{2r}{2r-1}}(\mathbb R^{d})).$ 
   
   \item Assume that $K$ is given by \eqref{di} and
   $u_{0}\in M^{1,1}(\R^d).$   Then there exists $T^{*}=T^{\ast}(\|u_{0}\|_{M^{1,1} })$ such that \eqref{HTEH} has a unique solution $u\in C([0, T^{*}), M^{1,1} (\mathbb R^d)).$  
\end{enumerate}  
\end{Theorem}

Now we briefly mention the mathematical literature. 
Carles-Mauser-Stimming \cite{rc3} and Cao-Carles \cite{pc} have studied  well-posedness  for  \eqref{HTEH} with Hartree type nonlinearity. Zhang \cite{jz} and  Carles \cite{rc, rc1, rc2} have studied  well-posedness  for  \eqref{HTEH} with power type nonlinearity.  We would like to mention  that  so far all previous  authors have studied \eqref{HTEH} in the energy space
\[\Sigma= \left\{ f\in \mathcal{S}'(\mathbb R^d): \|f\|_{\Sigma}:=\|f\|_{L^2} + \|\nabla f\|_{L^2} + \|xf\|_{L^2} < \infty \right\}.\]

Finally, we note that  the existence  of solution  for \eqref{HTEH}  is shown under very low regularity assumption for the  initial data. Specifically, to see how typical  Cauchy data  Theorems \ref{mt}, \ref{mtv}, and \ref{mtp} can handle, see  Proposition \ref{exa}  and Lemma \ref{rl} \eqref{el} below. Theorems \ref{mt}, \ref{mtv} \ref{mtp} and \ref{wht1}  highlights  that modulation spaces are  a  good alternative  as compared to Sobolev  and Besov  spaces to study equation \eqref{HTEH}. And we hope that our results will be useful  for the  further study (e.g., stability and scattering theory) of equation \eqref{HTEH}.

This paper is organized as follows.  In Section \ref{p}, we introduce  notations and preliminaries which will be used  in the sequel. In particular, in Subsections \ref{ns} and \ref{MW}, we introduce  $L^p_s-$Sobolev spaces and  modulation spaces(and their properties) respectively.  In Subsection \ref{whm}, we review boundedness of Schr\"odinger  propagator associated  with  harmonic potential on modulation spaces.  As an application of Strichartz's estimates and conservation of mass, we obtain  global well-posedness for \eqref{HTEH} in  $L^2(\mathbb R^d)$ in Section \ref{smt}. We shall see this will turn out to be  one of the main tools to obtain global well-posedness in modulation spaces.  In Subsections \ref{pmt}, \ref{pmtv}, \ref{phu} and \ref{lwhp},   we prove Theorems \ref{mt}, \ref{mtv}, \ref{mtp} and \ref{wht1} respectively. 

\section{Preliminaries}\label{p}  
\subsection{Notations} \label{ns} The notation $A \lesssim B $ means $A \leq cB$ for  some constant $c > 0 $. The symbol $A_{1}\hookrightarrow A_{2}$ denotes the continuous embedding  of the topological linear space $A_{1}$ into $A_{2}.$ 
If $\alpha= (\alpha_1,..., \alpha_d)\in \mathbb N^d$ is a multi-index, we set $|\alpha|= \sum_{j=1}^d \alpha_j, \alpha != \prod_{j=1}^d\alpha_j!.$ If $z=(z_1,...,z_d)\in \mathbb C^d,$ we put $z^{\alpha}=\prod_{j=1}^dz_{j}^{\alpha_j}$.The characteristic function of a set $E\subset \mathbb R^d$ is $\chi_{E}(x)=1$ if $x\in E$ and $\chi_E(x)=0$ if $x\notin E.$ Let $I\subset \mathbb R$ be an interval. Then the norm of the space-time Lebesgue space $L^{p}(I, L^q(\mathbb R^d))$ is defined by
$$\|u\|_{L^p(I, L^q(\mathbb R^d))}=\|u\|_{L^{p}_I L^q_x}= \left(\int_{I} \|u(t)\|^p_{L^q_x} dt \right)^{1/p}.$$
If there is no confusion, we simply write
$$\|u\|_{L^p(I, L^q)}=\|u\|_{L^{p}_I L^q_x}=\|u\|_{L^{p,q}_{t,x}}.$$
The Schwartz class is denoted  by $\mathcal{S}(\mathbb R^{d})$ (with its usual topology), and the space of tempered distributions is denoted by  $\mathcal{S'}(\mathbb R^{d}).$  For $x=(x_1,\cdots, x_d), y=(y_1,\cdots, y_d) \in \mathbb R^d, $ we put $x\cdot y = \sum_{i=1}^{d} x_i y_i.$ Let $\mathcal{F}:\mathcal{S}(\mathbb R^{d})\to \mathcal{S}(\mathbb R^{d})$ be the Fourier transform  defined by  
\begin{eqnarray*}
\mathcal{F}f(\xi)=\widehat{f}(\xi)=  (2\pi)^{-d/2} \int_{\mathbb R^{d}} f(x) e^{- i  x \cdot \xi}dx, \  \xi \in \mathbb R^d.
\end{eqnarray*}
Then $\mathcal{F}$ is a bijection  and the inverse Fourier transform  is given by
\begin{eqnarray*}
\mathcal{F}^{-1}f(x)=f^{\vee}(x)= (2\pi)^{-d/2}\int_{\mathbb R^{d}} f(\xi)\, e^{ i x\cdot \xi} d\xi,~~x\in \mathbb R^{d}.
\end{eqnarray*}
 The  Fourier transform can be uniquely extended to $\mathcal{F}:\mathcal{S}'(\mathbb R^d) \to \mathcal{S}'(\mathbb R^d).$  The \textbf{Fourier-Lebesgue spaces} $\mathcal{F}L^p(\mathbb R^d)$ is defined by 
$$\mathcal{F}L^p(\mathbb R^d)= \left\{f\in \mathcal{S}'(\mathbb R^d): \hat{f}\in L^{p}(\mathbb R^d) \right\}.$$
The  $\mathcal{F}L^{p} (\mathbb R^d)-$norm is denoted by
 $$\|f\|_{\mathcal{F}L^{p}}= \|\hat{f}\|_{L^{p}} \ (f\in \mathcal{F}L^{p}(\mathbb R^{d})).$$

The standard  Sobolev  spaces $W^{s,p}(\R^d) \ (1<p< \infty, s\geq 0)$ have a different character according to whether $s$ is integer or not. Namely, for $s$ integer, they consist of  $L^p-$functions with derivatives in  $L^p$ up to order $s$, hence coincide with the \textbf{$L^p_s-$Sobolev spaces }(also known as Bessel potential
spaces), defined for  $s\in \R$ by
$$L^p_s(\mathbb R^d)=\left\{f\in \mathcal{S}'(\mathbb R^d): \mathcal{F}^{-1} [\langle \cdot \rangle^s \mathcal{F}(f)] \in L^p(\mathbb R^d) \right\},$$
where  $\langle \xi \rangle^{s} = (1+ |\xi|^2)^{s/2} \ (\xi \in \mathbb R^d).$ Note that $L^p_{s_1}(\R^d) \hookrightarrow L^p_{s_2}(\R^d)$ if $s_2\leq s_1.$

\subsection{Modulation  Spaces}\label{MW}
In 1983, Feichtinger \cite{HG4} introduced a  class of Banach spaces,  the so called modulation spaces,  which allow a measurement of space variable and Fourier transform variable of a function or distribution on $\mathbb R^d$ simultaneously using the short-time Fourier transform(STFT). The  STFT  of a function $f$ with respect to a window function $g \in {\mathcal S}(\R^d)$ is defined by
$$  V_{g}f(x,y) = (2\pi)^{-d/2} \int_{\R^d} f(t) \overline{g(t-x)} \, e^{- i y\cdot t} \, dt,  ~  (x, y) \in \mathbb R^{2d} $$
whenever the integral exists. 
For $x, y \in \R^d$ the translation operator $T_x$ and the modulation operator $M_y$ are
defined by $T_{x}f(t)= f(t-x)$ and $M_{y}f(t)= e^{i y\cdot t} f(t).$ In terms of these
operators the STFT may be expressed as
\begin{eqnarray}
\label{ipform} V_{g}f(x,y)=\langle f, M_{y}T_{x}g\rangle\nonumber
\end{eqnarray}
 where $\langle f, g\rangle$ denotes the inner product for $L^2$ functions,
or the action of the tempered distribution $f$ on the Schwartz class function $g$.  Thus $V: (f,g) \to V_g(f)$ extends to a bilinear form on $\mathcal{S}'(\R^d) \times \mathcal{S}(\R^d)$ and $V_g(f)$ defines a uniformly continuous function on $\R^{d} \times \R^d$ whenever $f \in \mathcal{S}'(\R^d) $ and $g \in  \mathcal{S}(\R^d)$.
\begin{Definition}[modulation spaces]\label{ms} Let $1 \leq p,q \leq \infty,$ and $0\neq g \in{\mathcal S}(\R^d)$. The  modulation space   $M^{p,q}(\R^d)$
is defined to be the space of all tempered distributions $f$ for which the following  norm is finite:
$$ \|f\|_{M^{p,q}}=  \left(\int_{\R^d}\left(\int_{\R^d} |V_{g}f(x,y)|^{p} dx\right)^{q/p} \, dy\right)^{1/q},$$ for $ 1 \leq p,q <\infty$. If $p$ or $q$ is infinite, $\|f\|_{M^{p,q}}$ is defined by replacing the corresponding integral by the essential supremum. 
\end{Definition}
\begin{Remark}\label{CR}
\begin{enumerate} We note following
\item \label{equidm}
The definition of the modulation space given above, is independent of the choice of 
the particular window function.  See \cite[Proposition 11.3.2(c), p.233]{gro}, \cite{baob}.

\item \label{msw}
For a pair of functions $f$ and the Gaussian $ \Phi_0(\xi) = \pi^{-d/2} e^{-\frac{1}{2}|\xi|^2} $, 
the Fourier-Wigner transform of $ f $  and  $\Phi_0$ is defined by 
$$ F(x,y):=\langle \pi(x+iy)f,\Phi_{0}\rangle  = \int_{\mathbb R^d} e^{i(x \cdot \xi+\frac{1}{2}x\cdot y)} f(\xi+y)\Phi_0(\xi) d\xi. $$
We say $f\in M^{q,p}(\mathbb R^d)$ if  $ \left\| \|F(x,y)\|_{L_y^{q}}\right\|_{L^p_x}<\infty.$
\end{enumerate}
\end{Remark}
Next we justify Remark  \ref{CR}\eqref{msw}: we shall see how the Fourier-Wigner transform  and the STFT are related. We consider the Heisenberg group $\mathbb H^{d}=\mathbb C^{d} \times \mathbb R$ with the group law
$$(z,t)(w,s)=\left(z+w, t+s+\frac{1}{2}  \text{Im} (z\cdot \bar{w})\right).$$
Let $\pi$ be the Schr\"odinger representation of the Heisenberg group  which is realized on $L^{2}(\mathbb R^d)$ and explicitly given by 
$$\pi (x,y,t) \phi (\xi)=e^{it} e^{i(x\cdot \xi + \frac{1}{2}x\cdot y)} \phi (\xi +y)$$
where $x,y \in \mathbb R^d, t \in \mathbb R, \phi \in L^{2}(\mathbb R^d).$ The Fourier-Wigner transform of two functions $f, g\in L^{2}(\mathbb R^d)$ is defined by 
$$W_{g}f(x,y)= (2\pi)^{-d/2} \langle \pi (x,y, 0)f, g \rangle.$$
We recall  polarised Heisenberg group
$ \mathbb H^d_{pol} $ which is just $ \mathbb R^d \times \mathbb R^d \times \mathbb R $ with the group law
$$ (x,y,t)(x',y',t') = (x+x',y+y', t+t'+x'\cdot y)$$
and the representation    $\rho (x,y, e^{it})$ acting on  $L^{2}(\mathbb R^d)$ is given by 
$$\rho(x,y, e^{it})\phi (\xi)= e^{it} e^{ix\cdot \xi} \phi (\xi +y), \ \  \phi \in L^{2}(\mathbb R^d).$$
We now write the Fourier-Wigner transform  in terms of  the STFT:  Specifically,
we put $\rho(x,y)\phi(\xi)= e^{ix\cdot \xi} \phi (\xi +y), z= (x,y),$ and  we  have 
\begin{eqnarray}\label{ui}
\langle \pi(z)f, g\rangle=\langle \pi(x,y)f, g \rangle = e^{\frac{i}{2}x\cdot y} \langle \rho(x,y)f, g\rangle= e^{-\frac{i}{2} x\cdot y} V_{g}f(y,-x).
\end{eqnarray}
This useful identity \eqref{ui} reveals that  the definition of modulation spaces we have introduced  in  Remark \ref{CR}\eqref{msw} and  Definition \ref{ms}  is essentially the same.

The following basic properties of modulation spaces are well-known   and for the proof we refer  to \cite{gro, HG4, baob}.

\begin{Lemma}\label{rl} Let $p,q, p_{i}, q_{i}\in [1, \infty]$  $(i=1,2), s_1, s_2 \in \mathbb R.$ Then
\begin{enumerate}
\item \label{ir} $M_{s_1}^{p_{1}, q_{1}}(\mathbb R^{d}) \hookrightarrow M_{s_2}^{p_{2}, q_{2}}(\mathbb R^{d})$ whenever $p_{1}\leq p_{2},q_{1}\leq q_{2},$ and $ s_2 \leq s_1.$

\item \label{el} $M^{p,q_{1}}(\mathbb R^{d}) \hookrightarrow L^{p}(\mathbb R^{d}) \hookrightarrow M^{p,q_{2}}(\mathbb R^{d})$ holds for $q_{1}\leq \text{min} \{p, p'\}$ and $q_{2}\geq \text{max} \{p, p'\}$ with $\frac{1}{p}+\frac{1}{p'}=1.$
\item \label{rcs} $M^{\min\{p', 2\}, p}(\mathbb R^d) \hookrightarrow \mathcal{F} L^{p}(\mathbb R^d)\hookrightarrow M^{\max \{p',2\},p}(\mathbb R^d),  \frac{1}{p}+\frac{1}{p'}=1.$
\item \label{d} $\mathcal{S}(\mathbb R^{d})$ is dense in  $M^{p,q}(\mathbb R^{d})$ if $p$ and $q<\infty.$
\item \label{fi} The Fourier transform $\mathcal{F}:M^{p,p}(\mathbb R^{d})\to M^{p,p}(\mathbb R^{d})$ is an isomorphism.
\item The space  $M^{p,q}(\mathbb R^{d})$ is a  Banach space.
\item \label{ic}The space $M^{p,q}(\mathbb R^{d})$ is invariant under complex conjugation.
\end{enumerate}
\end{Lemma}
\begin{proof}
For the proof of statements  \eqref{ir}, \eqref{el}, \eqref{rcs} and \eqref{d},  see  \cite[Theorem 12.2.2]{gro}, \cite[Proposition 1.7]{JT},  \cite[Corollary 1.1]{cks} and \cite[Proposition 11.3.4]{gro} respectively.  The proof  for the statement (5)  can be derived from  the fundamental identity of time-frequency analysis: 
$$ V_gf(x, w) = e^{-2 \pi i x \cdot w } \, V_{\widehat{g}} \widehat{f}(w, -x),$$
which is easy to obtain. The proof of the statement (7) is trivial, indeed, we have $\|f\|_{M^{p,q}}=\|\bar{f}\|_{M^{p,q}}.$
\end{proof}
\begin{Theorem}[Algebra property] \label{pl} Let $p,q, p_{i}, q_{i}\in [1, \infty]$  $(i=0,1,2).$
\begin{enumerate}
\item \label{mp}
If   $\frac{1}{p_1}+ \frac{1}{p_2}= \frac{1}{p_0}$ and $\frac{1}{q_1}+\frac{1}{q_2}=1+\frac{1}{q_0}, $ then
\begin{eqnarray*}\label{prm}
M^{p_1, q_1}(\mathbb R^{d}) \cdot M^{p_{2}, q_{2}}(\mathbb R^{d}) \hookrightarrow M^{p_0, q_0}(\mathbb R^{d})
\end{eqnarray*}
with norm inequality $\|f g\|_{M^{p_0, q_0}}\lesssim \|f\|_{M^{p_1, q_1}} \|g\|_{M^{p_2,q_2}}.$
In particular, the  space $M^{p,q}(\mathbb R^{d})$ is a poinwise $\mathcal{F}L^{1}(\mathbb R^{d})$-module, that is, it satisfies
\begin{eqnarray*}
\|fg\|_{M^{p,q}} \lesssim \|f\|_{\mathcal{F}L^{1}} \|g\|_{M^{p,q}}.
\end{eqnarray*} 
\item  If $\frac{1}{p_1}+ \frac{1}{p_2}=1+\frac{1}{p_0}$ and $\frac{1}{q_1}+\frac{1}{q_2}=1,$ then
$$M^{p_{1},q_{1}}(\mathbb R^{d}) \ast M^{p_2, q_2}(\mathbb R^{d}) \hookrightarrow M^{p_0, q_0}(\mathbb R^{d}) $$  with norm inequality $\|f\ast h\|_{M^{p_0, q_0}} \lesssim \|f\|_{M^{p_1,q_1}} \|h\|_{M^{p_2, q_2}}.$ In particular, $M^{p,q}(\mathbb R^{d})$ is a left  Banach $L^{1}(\mathbb R^{d})-$ module with respect to  convolution.
\end{enumerate}
\end{Theorem}
\begin{proof}
For the proof of statements \eqref{mp}, we refer to  \cite{ambenyi} and \cite{JT} respectively.
\end{proof}

We remark that there is also an equivalent  definition of modulation spaces using frequency-uniform decomposition techniques (which is quite similar in the spirit of Besov spaces  $B^{p,q}(\R^d)$), independently  studied by Wang et al. in \cite{bao, wang2}, which has turned out to be very fruitful in PDE.  Since  Besov and Sobolev  spaces are widely used  PDE, we recall  some inclusion relations between Besov, Sobolev and  modulation spaces. Besides, this gives us a  flavor how far low regularity initial Cauchy data  in Theorems \ref{mt}, \ref{mtv},  \ref{mtp}, and \ref{wht1} one can take.
\begin{Proposition}[examples]\label{exa} The following are true:
\begin{enumerate}

\item \label{msi} $L^p_s(\mathbb R^d) \hookrightarrow  M^{p,p}(\mathbb R^d) $ for $s\geq d(\frac{2}{p}-1)$ and $1<p\leq 2.$
\item  \label{kr}  $L^p_s(\R^d) \hookrightarrow M^{r,1}(\R^d)$ for $s>d, 1\leq p,q,r \leq \infty,$ and  $\frac{1}{p}+\frac{1}{q}=1+\frac{1}{r}.$
\item \label{tne} Define \[f(x)= \sum_{k\neq 0} |k|^{-\frac{d}{q}-\epsilon} e^{ik\cdot x} e^{-|x|^2} \ \text{in}  \ \mathcal{S}'(\mathbb R^d). \]  Then $f \in M^{p,q}(\R^d)$ for  $1\leq p, q \leq \infty, (p,q) \neq (1, \infty), (\infty, 1).$ 
\end{enumerate}
\end{Proposition}

\begin{proof}
For the proof of statements \eqref{msi}, \eqref{kr},  see \cite[Theorem 3.1]{ks}, \cite[Theorem 3.4]{krp}.  For the proof of statement \ref{tne}, see \cite[Lemma 3.8]{msnt}.
\end{proof}
For further relations between modulation, Sobolev, and Besov spaces, we refer to \cite{msnt, krp, JT, ks, cks, bao, wang2, baob} and the references therein.
\subsection{Hermite and special Hermite functions}\label{whm} The spectral decomposition of $H=-\Delta + |x|^2$ is given by  the Hermite expansion. Let $\Phi_{\alpha}(x),  \  \alpha \in \mathbb N^d$ be the normalized Hermite functions which are products of one dimensional Hermite functions. More precisely, 
$ \Phi_\alpha(x) = \Pi_{j=1}^d  h_{\alpha_j}(x_j) $ 
where 
$$ h_k(x) = (\sqrt{\pi}2^k k!)^{-1/2} (-1)^k e^{\frac{1}{2}x^2}  \frac{d^k}{dx^k} e^{-x^2}.$$ The Hermite functions $ \Phi_\alpha $ are eigenfunctions of $H$ with eigenvalues  $(2|\alpha| + d)$  where $|\alpha |= \alpha_{1}+ ...+ \alpha_d.$ Moreover, they form an orthonormal basis for $ L^2(\R^d).$ The spectral decomposition of $ H $ is then written as
$$ H = \sum_{k=0}^\infty (2k+d) P_k,~~~~ P_kf(x) = \sum_{|\alpha|=k} \langle f,\Phi_\alpha\rangle \Phi_\alpha$$ 
where $\langle\cdot, \cdot \rangle $ is the inner product in $L^2(\mathbb{R}^d)$. Given a function $m$ defined and bounded  on the set of all natural numbers we can use the spectral theorem to define $m(H).$ The action of $m(H)$ on a function $f$ is given by
\begin{eqnarray}\label{dhm}
m(H)f= \sum_{\alpha \in \mathbb N^d} m(2|\alpha| +d) \langle f, \Phi_{\alpha} \rangle \Phi_{\alpha} = \sum_{k=0}^\infty m(2k+d)P_kf.
\end{eqnarray}
This operator  $m(H)$ is bounded on $L^{2}(\mathbb R^d).$ This follows immediately from the Plancherel theorem for the Hermite expansions as  $m$ is bounded.  On the other hand, the mere boundedness of $m$ is not sufficient  to imply  the $L^{p}$ boundedness of $m(H)$ for $p\neq 2$ (see \cite{ST}). 

In the sequel, we make use of some properties of special Hermite functions $ \Phi_{\alpha, \beta} $ which are defined as follows. We recall \eqref{ui}  and define
\begin{eqnarray}\label{SHF}
\Phi_{\alpha,\beta}(z) =  (2\pi)^{-d/2} \langle \pi(z)\Phi_\alpha,\Phi_\beta\rangle.
\end{eqnarray}
Then it is well known that these so called special Hermite functions form an orthonormal basis for $ L^2(\mathbb C^d).$ In particular, we have (\cite[Theorem 1.3.5]{ST})
\begin{eqnarray}\label{EPSHF}
 \Phi_{\alpha,0}(z) =  (2\pi)^{-d/2} (\alpha !)^{-1/2} \left(\frac{i}{\sqrt{2}} \right)^{|\alpha|}\bar{z}^{\alpha} e^{-\frac{1}{4} |z|^2}.
\end{eqnarray}

We define \textbf{Schr\"odinger propagator associated to harmonic oscillator $m(H)=e^{itH}$}, denoted by $U(t),$ by equation \eqref{dhm}  with $m(n)=e^{itn} \ (n\in \mathbb N, t\in  \mathbb R).$ Next proposition says that $U(t)$ is uniformly  bounded on $M^{p,p}(\R^d).$ Specifically, we have 

\begin{Theorem}[\cite{drt}]\label{mso}
The Schr\"odinger  propagator  $m(H)=U(t)=e^{itH}$ is bounded on $M^{p,p}(\mathbb R^d) \ (t\in \mathbb R)$ for all $1\leq p < \infty.$ In fact, we have $\|e^{itH}f\|_{M^{p,p}}=\|f\|_{M^{p,p}}.$
\end{Theorem}
\begin{proof}Let $f\in \mathcal{S}(\mathbb R^d).$
Then we have the  Hermite expansion of  $f$ as follows:
\begin{eqnarray}\label{he}
f= \sum_{\alpha \in \mathbb N^d} \langle f, \Phi_{\alpha}\rangle \Phi_{\alpha}.
\end{eqnarray} 
Now using \eqref{he} and \eqref{SHF}, we obtain
\begin{eqnarray}\label{pihe}
\langle \pi(z)f, \Phi_0 \rangle & =  & \sum_{\alpha \in \mathbb N^d}\langle f, \Phi_{\alpha} \rangle \langle \pi(z)\Phi_{\alpha}, \Phi_{0} \rangle \nonumber\\
& = & \sum_{\alpha \in \mathbb N^d} \langle f, \Phi_{\alpha} \rangle \Phi_{\alpha, 0}(z).
\end{eqnarray}
Since $\{\Phi_{\alpha} \}$ forms an orthonormal basis for $L^{2}(\mathbb R^d),$ \eqref{pihe} gives
\begin{eqnarray}\label{ff}
\langle \pi(z)m(H)f, \Phi_0  \rangle  & = &  \sum_{\alpha \in \mathbb N^d} \langle m(H)f, \Phi_{\alpha} \rangle \Phi_{\alpha, 0}(z) \nonumber \\
& = & \sum_{\alpha \in \mathbb N^d} m(2|\alpha| +d) \langle f, \Phi_{\alpha} \rangle \Phi_{\alpha, 0}(z).\nonumber
\end{eqnarray}
Therefore, for $m(H) = e^{itH}$, we have \begin{eqnarray}\label{1}
\langle \pi(z)e^{itH}f, \Phi_0  \rangle  &=  &e^{itd} \sum_{\alpha \in \mathbb N^d} e^{2it|\alpha|} \langle f, \Phi_{\alpha} \rangle \Phi_{\alpha, 0}(z) \nonumber \\
&= &e^{itd} (2\pi)^{-d/2}  \sum_{\alpha \in \mathbb N^d} e^{2it|\alpha|} \langle f, \Phi_{\alpha} \rangle  (\alpha !)^{-1/2} \left(\frac{i}{\sqrt{2}} \right)^{|\alpha|}\bar{z}^{\alpha} e^{-\frac{1}{4} |z|^2}\nonumber \\
&= & e^{itd} (2\pi)^{-d/2}  \sum_{\alpha \in \mathbb N^d} C_{\alpha}(f)e^{2it|\alpha|} \bar{z}^{\alpha} e^{-\frac{1}{4} |z|^2}
\end{eqnarray}
where $ C_{\alpha}(f):= \langle f, \Phi_{\alpha} \rangle  (\alpha !)^{-1/2} \left(\frac{i}{\sqrt{2}} \right)^{|\alpha|}.$  In view of \eqref{ui} and \eqref{1}, we have 
\begin{align}\label{o}
\|e^{itH}f\|^p_{M^{p,p}}&=  \frac{1}{(2\pi)^{d/2}}\int_{\mathbb{C}^d}  \left |  \sum_{\alpha \in \mathbb N^d} C_{\alpha}(f)e^{2it|\alpha|} \bar{z}^{\alpha} e^{-\frac{1}{4} |z|^2}\, \right |^p dz.
\end{align} 
By using polar coordinates  $z_j= r_j e^{i \theta_j}$,  $r_j:=|z_j|\in [0, \infty), z_j\in \mathbb C$ and  $\theta_j \in [0, 2\pi),$ we get
\begin{eqnarray}\label{LPC}
z^{\alpha}= r^{\alpha} e^{i \alpha \cdot \theta} \ \  \text{and} \ \  dz= r_1 r_2 \cdots r_d d\theta dr
\end{eqnarray}
 where $r = (r_1,\cdots, r_d)$, $\theta = (\theta_1,\cdots,\theta_d)$, $dr = dr_1\cdots dr_d$, $d\theta = d\theta_1\cdots d\theta_d, |r|=\sqrt{\sum_{j=1}^d r_j^2}.$ \\ 
By writing the integral over $\mathbb C^d= \mathbb R^{2d}$ in polar coordinates in
each time-frequency pair and using  \eqref{LPC}, we have
\begin{eqnarray}\label{h} \int_{\mathbb{C}^d}  \left |  \sum_{\alpha \in \mathbb N^d} C_{\alpha}(f)e^{2it|\alpha|} \bar{z}^{\alpha} e^{-\frac{1}{4} |z|^2}\, \right |^p dz&\\
=\prod_{j=1}^d
 \int_{\mathbb{R}^+}\int_{[0,2\pi]} \left|\sum_{\alpha  \in \mathbb N^d } C_{\alpha}(f)  r^{\alpha}  e^{i \sum_{j=1}^d(2t-\theta_j) \alpha_j} e^{-\frac{1}{4} |r|^2} \right |^p r_j  dr_j d\theta_j.  \nonumber 
\end{eqnarray}
By a simple change of variable $(\theta_j-2t)\rightarrow \theta_j$, we obtain
\begin{eqnarray}\label{b}  \prod_{j=1}^d
 \int_{\mathbb{R}^+}\int_{[0,2\pi]} \left|\sum_{\alpha \in \mathbb N^d } C_{\alpha}(f)  r^{\alpha} e^{i \sum_{j=1}^d(2t-\theta_j) \alpha_j} e^{-\frac{1}{4} |r|^2} \right |^p r_j  dr_j d\theta_j  \\
 = \prod_{j=1}^d \int_{\mathbb{R}^+}\int_{[0,2\pi]} \left|\sum_{\alpha \in \mathbb N^d } C_{\alpha}(f)  r^{\alpha} e^{-i \theta \cdot \alpha} e^{-\frac{1}{4} |r|^2} \right |^p r_j  dr_j d\theta_j . \nonumber
\end{eqnarray}
Combining  \eqref{o}, \eqref{h},  \eqref{b}, and Lemma \ref{pl}\eqref{d},  we have $\|e^{itH}f \|_{M^{p,p}} = \|f\|_{M^{p,p}}$ for $f\in M^{p,p}(\mathbb R^d).$ 
\end{proof}
\begin{Remark}\label{nrr} In view of \eqref{pihe} and \eqref{1},  in Theorem \ref{mso}, we cannot expect to replace $M^{p,p}$ norm by $M^{p,q}$ norm. See also \cite{kks}.
\end{Remark}

\section{Global wellposedness in $L^2(\mathbb R^d)$}\label{smt}
In this section we prove global well-posedness for  \eqref{HTEH} with Cauchy data in  $L^2(\mathbb R^d).$   To this end, we need Strichartz estimates, and hence we recall it.
\begin{Definition} A pair $(q,r)$ is admissible if  $2\leq r< \frac{2d}{d-2}$ ($2\leq r \leq \infty$ if $d=1$ and  $2\leq r < \infty$ if $d=2$)
$$\frac{2}{q} =  d \left( \frac{1}{2} - \frac{1}{r} \right).$$ 
\end{Definition}

\begin{Proposition}[\cite{rc}] \label{seh}For any time slab $I$ and  admissible pairs $(p_i, q_i) (i=1,2):$
\begin{enumerate}
\item There exists $C_r$ such that 
$$\|U(t) \phi  \|_{L^{q,r}_{t,x}} \leq C_r \|\phi \|_{L^2}, \   \forall \phi \in L^2 (\mathbb R^d).$$
\item Define 
$$DF(t,x) = \int_0^t U(t-\tau )F(\tau,x) d\tau.$$
There exists  $C=C(|I|, q_1)$ (constant) such that for all intervals $I \ni 0, $
$$ \|D(F)\|_{L^{p_1, q_1}_{t,x}}  \leq C  \|F\|_{L^{p_2', q_2'}_{t,x}}, \ \forall F \in L^{p_2'} (I, L^{q_2'})$$  where $p_i'$ and $ q_i'$ are H\"older conjugates of $p_i$ and $q_i$
respectively.
\end{enumerate}
\end{Proposition}
We also need to work with the  convolution with the Hartree potential  $|x|^{-\gamma},$  so for the convenience of reader we recall:
\begin{Proposition}[Hardy-Littlewood-Sobolev inequality] \label{hls}Assume that  $0<\gamma< d$ and $1<p<q< \infty$ with
$$\frac{1}{p}+\frac{\gamma}{d}-1= \frac{1}{q}.$$
Then the map $f \mapsto |x|^{-\gamma}\ast f$ is bounded from $L^p(\mathbb R^d)$ to $L^q(\mathbb R^d):$
$$\||x|^{-\gamma}\ast f\|_{L^q} \leq C_{d,\gamma, p} \|f\|_{L^p}.$$
\end{Proposition}
\begin{Proposition}\label{miD}Let    $F(u)$   and  $K$ be given by  \eqref{htn} and \eqref{hk}  respectively with $k=1,$ and 
 $ 0<\gamma < \text{min} \{2, d/2\}, d\in \mathbb N$. If $u_{0}\in L^{2}(\mathbb R^{d}),$ then \eqref{HTEH} has a unique global solution 
$$u\in C([0, \infty), L^{2}(\mathbb R^d))\cap L^{8/\gamma}_{loc}([0, \infty), L^{4d/(2d-\gamma)} (\mathbb R^d)).$$ 
In addition, its $L^{2}-$norm is conserved, 
$$\|u(t)\|_{L^{2}}=\|u_{0}\|_{L^{2}}, \   \forall t \in \mathbb R,$$
and for all  admissible pairs  $(p,q), u \in L_{loc}^{p}(\mathbb R, L^{q}(\mathbb R^d)).$
\end{Proposition}
\begin{proof} By Duhamel's formula, we write \eqref{HTEH}
as 
$$u(t)=U(t)u_0- i \int_0^t U(t-\tau) (K \ast |u|^2)u(\tau) d\tau:= \Phi(u)(t).$$  We introduce the space
\begin{eqnarray*}
Y(T)  & =  &\{ \phi \in C\left([0,T], L^2(\mathbb R^d) \right): \|\phi \|_{L^{\infty}([0, T], L^2)}  \leq 2 \|u_0\|_{L^2}, \\
&&  \|\phi\|_{L^{\frac{8}{\gamma}} ([0,T], L^{\frac{4d}{2d-\gamma}}  ) } \lesssim \|u_0\|_{L^2}\}
\end{eqnarray*}
and the distance 
$$d(\phi_1, \phi_2)= \|\phi_1 - \phi_2 \|_{L^{\frac{8}{\gamma} }\left( [0, T], L^{\frac{4d}{(2d- \gamma)}}\right)}.$$ Then $(Y, d)$ is a complete metric space. Now we show that $\Phi$ takes $Y(T)$ to $Y(T)$ for some $T>0.$
We put 
$$ \ q= \frac{8}{\gamma}, \  r= \frac{4d}{2d- \gamma}.$$ 
Note that $(q,r)$ is  admissible and 
$$ \frac{1}{q'}= \frac{4- \gamma}{4} + \frac{1}{q}, \  \frac{1}{r'}= \frac{\gamma}{2d} + \frac{1}{r}.$$
Let $(\bar{q}, \bar{r}) \in \{ (q,r), (\infty, 2) \}.$ By Proposition \ref{seh} and  H\"older inequality, we have 
\begin{eqnarray*}
\|\Phi(u)\|_{L_{t,x}^{\bar{q}, \bar{r}}} &  \lesssim &  \|u_0\|_{L^2} + \|(K \ast |u|^2)u \|_{L_{t,x}^{q',r'}}\\
& \lesssim &  \|u_0\|_{L^2} + \|K \ast |u|^2\|_{L_{t,x}^{\frac{4}{4-\gamma}, \frac{2d}{\gamma}}} 
\|u\|_{L^{q,r}_{t,x}}.
\end{eqnarray*}
Since $0<\gamma< \min \{2, d \}$, by Proposition \ref{hls},  we  have 
\begin{eqnarray*}
\|K \ast |u|^2\|_{L_{t,x}^{\frac{4}{4-\gamma}, \frac{2d}{\gamma}}}  & = &  \left\|  \|K\ast |u|^2\|_{L_x^{\frac{2d}{\gamma}}} \right\|_{L_t^{\frac{4}{4- \gamma}}}\\
& \lesssim &   \left \| \||u|^2\|_{L_x^{\frac{2d}{2d- \gamma}}} \right\|_{L_t^{\frac{4}{4- \gamma}}} \\
& \lesssim & \|u\|^2_{L_{t,x}^{\frac{8}{4- \gamma},r}}\\
& \lesssim & T^{1- \frac{\gamma}{2}} \|u\|^2_{L_{t,x}^{q,r}}.
\end{eqnarray*}
(In the last inequality we have used inclusion relation for the $L^p$ spaces on finite measure spaces: $\|\cdot\|_{L^p(X)} \leq \mu(X)^{\frac{1}{p}-\frac{1}{q}} \|\cdot \|_{L^{q}(X)}$ if measure of $X$:$\mu(X)<\infty, 0<p<q<\infty$.) 
Thus, we  have 
$$
\|\Phi(u)\|_{L_{t,x}^{\bar{q}, \bar{r}}}   \lesssim   \|u_0\|_{L^2}+T^{1- \frac{\gamma}{2}} \|u\|^3_{L_{t,x}^{q,r}}.$$
This shows that $\Phi$ maps $Y(T)$ to $Y(T).$  Next, we show $\Phi$
is a  contraction. For this, as 
 calculations performed  before, first we note  that 
\begin{eqnarray}\label{mi}
 \|(K\ast |v|^{2})(v-w)\|_{L_{t,x}^{q',r'}} \lesssim  T^{1-\frac{\gamma}{2}} \|v\|^2_{L_{t,x}^{q,r}} \|v-w\|_{L_{t,x}^{q,r}}.
\end{eqnarray}
Put $\delta = \frac{8}{4-\gamma}$  and notice that \[\frac{1}{q'}= \frac{1}{2}+ \frac{1}{\delta}, \frac{1}{r'}= \frac{1}{r}+ \frac{\gamma}{2d},  \frac{1}{2}= \frac{1}{\delta} + \frac{1}{q}.\] Now using H\"older inequality, we obtain
\begin{eqnarray}\label{mi1}
 \|(K \ast (|v|^{2}- |w|^{2}))w\|_{L_{t,x}^{q',r'}} & \lesssim & \| K\ast \left( |v|^2-|w|^2\right)\|_{L_{t,x}^{2, \frac{2d}{\gamma}}} \|w\|_{L_{t,x}^{\delta, r}} \nonumber \\
 & \lesssim & \left( \|K \ast (v (\bar{v}- \bar{w})) \|_{L_{t,x}^{2, \frac{2d}{\gamma}}}\right. \nonumber
 \\
 && \left.+ \|K \ast \bar{w}(v-w)) \|_{L_{t,x}^{2, \frac{2d}{\gamma}}} \right) \|w\|_{L_{t,x}^{\delta,r}} \nonumber\\
 & \lesssim & \left( \|v\|_{L_{t,x}^{\delta,r}} \|w\|_ {L_{t,x}^{\delta,r}} +\|w\|^2_ {L_{t,x}^{\delta,r}}\right) \|v-w\|_{L_{t,x}^{q,r}}\nonumber\\
 & \lesssim & T^{1-\frac{\gamma}{2}}  \left( \|v\|_{L_{t,x}^{q,r}} \|w\|_ {L_{t,x}^{q,r}} +\|w\|^2_ {L_{t,x}^{q,r}}\right)  \|v-w\|_{L_{t,x}^{q,r}}.
\end{eqnarray}
In view of the identity
$$(K\ast |v|^{2})v- (K\ast |w|^{2})w= (K\ast |v|^{2})(v-w) + (K \ast (|v|^{2}- |w|^{2}))w, $$  \eqref{mi}, and \eqref{mi1} gives
\begin{eqnarray*}
\|\Phi(v)- \Phi(w)\|_{L_{t,x}^{q,r}} & \lesssim &  \|(K\ast |v|^{2})(v-w)\|_{L_{t,x}^{q',r'}} + \|(K \ast (|v|^{2}- |w|^{2}))w\|_{L_{t,x}^{q',r'}}\\
& \lesssim &   T^{1-\frac{\gamma}{2}}  \left( \|v\|^2_{L_{t,x}^{q,r}}  +\|v\|_{L_{t,x}^{q,r}} \|w\|_ {L_{t,x}^{q,r}} +\|w\|^2_ {L_{t,x}^{q,r}}\right)\\
&&  \|v-w\|_{L_{t,x}^{q,r}}.
\end{eqnarray*}
Thus $\Phi$ is a contraction form $Y(T)$ to $Y(T)$  provided that $T$ is sufficiently small. Then there exists a unique $u \in Y(T)$ solving \eqref{HTEH}. The global  existence of the solution \eqref{HTEH} follows from the conservation of the $L^2-$norm of $u.$ The last property of the proposition then follows from the Strichartz estimates applied with an arbitrary  admissible pair on the left hand side and the same pairs as above on the right hand side.
\end{proof}
\section{Proof of the Main Results}
\subsection{Global well-posedness in $M^{p,p}$ for the Hartree potential}\label{pmt}
In this section, we shall prove Theorem \ref{mt}.
For  convenience of  the reader, we recall
\begin{Proposition}[\cite{baho}]\label{fc} 
Let $d\geq 1,$  $0<\gamma <d$ and $\lambda \in \mathbb R.$ There exists $C=C(d,\gamma)$ such that the Fourier transform of $K$ defined by \eqref{hk} is
\begin{eqnarray*}
\widehat{K}(\xi)= \frac{\lambda C}{|\xi|^{d-\gamma}}.
\end{eqnarray*}
\end{Proposition}
We start with decomposing the Fourier transform of Hartree potential into Lebesgue spaces: indeed, in view of Proposition \ref{fc}, we  have
\begin{eqnarray}\label{dc}
\widehat{K}=k_1+k_2 \in L^{p}(\mathbb R^{d})+ L^{q}(\mathbb R^{d}),
\end{eqnarray}
where  $k_{1}:= \chi_{\{|\xi|\leq 1\}}\widehat{K} \in L^{p}(\mathbb R^{d})$ for all $p\in [1, \frac{d}{d-\gamma})$ and $k_{2}:= \chi_{\{|\xi|>1\}} \widehat{K} \in L^{q}(\mathbb R^{d})$ for all $q\in (\frac{d}{d-\gamma}, \infty].$
\begin{Lemma} [\cite{sdgb}]\label{iml} Let  $K$ be given by \eqref{hk} with $\lambda \in \mathbb R,$ and $ 0<\gamma < d$, and  $1\leq p \leq 2, 1\leq q < \frac{2d}{d+\gamma}$. Then for any $f\in M^{p,q}(\mathbb R^{d}),$ we have
\begin{eqnarray}\label{d1}
\|(K\ast |f|^{2}) f\|_{M^{p,q}} \lesssim\|f\|_{M^{p,q}}^{3}.
\end{eqnarray}
\end{Lemma}
\begin{proof}
By \eqref{mp}  and  \eqref{dc}, we have
\begin{eqnarray}\label{md1}
\|(K\ast |f|^{2}) f)\|_{M^{p,q}} & \lesssim  & \|K\ast |f|^{2}\|_{\mathcal{F}L^{1}} \|f\|_{M^{p,q}}\nonumber \\
& \lesssim & \left(\|k_{1} \widehat{|f|^{2}}\|_{L^{1}} + \|k_{2} \widehat{|f|^{2}}\|_{L^{1}} \right) \|f\|_{M^{p,q}}. 
\end{eqnarray}
Note that 
\begin{eqnarray}.\label{md2}
\|k_1 \widehat{|f|^2}\|_{L^1}  & \lesssim &  \|k_1\|_{L^1} \|\widehat{|f|^2}\|_{L^{\infty}} \nonumber \\
& \lesssim & \||f|^2\|_{L^1}= \|f\|_{L^2}^2 \nonumber \\
& \lesssim & \|f\|_{M^{p,q}}^2.
\end{eqnarray}
Let $1< \frac{d}{d-\gamma}<r\leq 2, \frac{1}{r}+ \frac{1}{r'}=1.$  Note that $\frac{1}{r'}+1= \frac{1}{r_1}+ \frac{1}{r_2},$ where $r_1=r_2:= \frac{2r}{2r-1} \in [1,2]$, and $r_1' \in [2, \infty]$ where $\frac{1}{r_1}+ \frac{1}{r_1^{'}}=1.$ 
Now  using  Young's inequality for convolution, Lemma \ref{rl} \eqref{rcs}, Lemma \ref{rl} \eqref{ir}, and Lemma \ref{rl} \eqref{ic},  we obtain
\begin{eqnarray}
\|k_2\widehat{|f|^2} \|_{L^1} & \leq & \|k_2\|_{L^r}  \|\widehat{|f|^2}\|_{L^{r'}} \nonumber\\
& \lesssim &  \|\hat{f} \ast \hat{\bar{f}}\|_{L^{r'}}\nonumber\\
& \lesssim & \|\hat{f}\|_{L^{r_1}} \|\hat{\bar{f}}\|_{L^{r_1}} \nonumber \\
& \lesssim &  \|f\|^2_{M^{\min\{r_1', 2\}, r_1}}  \lesssim \|f\|^2_{M^{2, r_1}}\nonumber
\end{eqnarray}
Since  $f:[\frac{d}{d-\gamma}, \infty]\to \mathbb R, f(r) = \frac{2r}{2r-1}$ is a decreasing function, by  Lemma \ref{rl}\eqref{ir}, we have
\begin{eqnarray}\label{md3}
\|k_2 \widehat{|f|^2}\|_{L^1}\lesssim  \|f\|^2_{M^{2, r_1}} \lesssim \|f\|^2_{M^{p,q}}. 
\end{eqnarray}
Combining \eqref{md1}, \eqref{md2}, and  \eqref{md3}, we obtain \eqref{d1}.
\end{proof}
\begin{Lemma} [\cite{sdgb}]\label{cl}
Let  $0<\gamma <d,$ and $1\leq p \leq 2, 1\leq q < \frac{2d}{d+\gamma}.$ For any $ f,g \in M^{p,q}(\mathbb R^{d})$, we have
$$\| (K\ast |f|^{2})f - (K\ast |g|^{2})g\|_{M^{p,q}} \lesssim  (\|f\|_{M^{p,q}}^{2}+\|f\|_{M^{p,q}}\|g\|_{M^{p,q}}+ \|g\|_{M^{p,q}}^{2}) \|f-g\|_{M^{p,q}}.$$
\end{Lemma}
\begin{proof}
By exploiting the ideas from the proof of Lemma \ref{iml}, we obtain
\begin{eqnarray}
\|(K\ast |f|^{2})(f-g)\|_{M^{p,q}} & \lesssim & \|K\ast |f|^{2}\|_{\mathcal{F}L^{1}} \|f-g\|_{M^{p,q}}\nonumber \\
& \lesssim &  \left( \|k_{1} \widehat{|f|^{2}}\|_{L^{1}} + \|k_{2} \widehat{|f|^{2}}\|_{L^{1}}  \right) \|f-g\|_{M^{p,p}}\nonumber \\
& \lesssim &  \|f\|^{2}_{M^{p,q}} \|f-g\|_{M^{p,q}}.\label{a1}
\end{eqnarray}
Let $1\leq s < \frac{d}{d-\gamma}<t\leq 2, \frac{1}{s}+\frac{1}{s'}=1, \frac{1}{t}+\frac{1}{t'}=1.$
We note that
\begin{eqnarray}
\|(K \ast (|f|^{2}- |g|^{2}))g\|_{M^{p,q}}  & \lesssim & \|K \ast (|f|^{2}-|g|^{2})\|_{\mathcal{F}L^{1}} \|g\|_{M^{p,q}}\nonumber \\
& \lesssim & \left( \|k_1\|_{L^s}   \| \widehat{ |f|^{2}- |g|^{2}}\|_{L^{s'}} + \|k_2\|_{L^t} \| \widehat{|f|^{2}- |g|^{2}}\|_{L^{t'}} \right) \nonumber \\
&&  \|g\|_{M^{p,p}}\nonumber \\ 
& \lesssim &  \left( \| \widehat{ |f|^{2}- |g|^{2}}\|_{L^{s'}} + \| \widehat{|f|^{2}- |g|^{2}}\|_{L^{t'}} \right)  \|g\|_{M^{p,q}}\label{a2}
\end{eqnarray}
Let $1\leq r \leq 2,$ and  $\frac{1}{r}+ \frac{1}{r'}=1.$  Note that $\frac{1}{r'}+1= \frac{1}{r_1}+ \frac{1}{r_2},$ where $r_1=r_2:= \frac{2r}{2r-1} \in [1,2]$, and $r_1' \in [2, \infty]$ where $\frac{1}{r_1}+ \frac{1}{r_1^{'}}=1.$   Now  using  Young's inequality for convolution,  and  exploiting ideas performed as in the proof of Lemma \ref{iml}, we obtain 
\begin{eqnarray*}
\| |f|^{2}- |g|^{2}\|_{ \mathcal{F}L^{r'}} & \lesssim  & \|(f-g)\bar{f} \|_{\mathcal{F}L^{r'}} + \| g(\bar{f}-\bar{g}\|_{\mathcal{F}L^{r'}}\\
& = & \|\widehat{(f-g)} \ast \hat{\bar{f}}\|_{L^{r'}} + \|\hat{g} \ast  \widehat{\overline{f-g}}\|_{L^{r'}}\\
& \lesssim &  \|f-g\|_{\mathcal{F}L^{r_1}}\|\bar{f}\|_{\mathcal{F}L^{r_1}} +  \|g\|_{\mathcal{F}L^{r_1}} \|\overline{f-g}\|_{\mathcal{F}L^{r_1}}\\
& \lesssim &  (\|f\|_{M^{2, r_1}} + \|g\|_{M^{2, r_1}})  \|f-g\|_{M^{2, r_1}}\\
& \lesssim & (\|f\|_{M^{p,q}} + \|g\|_{M^{p,q}})  \|f-g\|_{M^{p,q}}.
\end{eqnarray*}
Using this,  \eqref{a2} gives
\begin{eqnarray}
\|(K \ast (|f|^{2}- |g|^{2}))g\|_{M^{p,q}} &  \lesssim & \left( \|f\|_{M^{p,q}} + \|g\|_{M^{p,q}} \right)  \|g\|_{M^{p,q}}\nonumber \\
&&\|f-g\|_{M^{p,q}}\label{al}
\end{eqnarray}
Now taking the identity
\begin{eqnarray}\label{inpc}
(K\ast |f|^{2})f- (K\ast |g|^{2})g= (K\ast |f|^{2})(f-g) + (K \ast (|f|^{2}- |g|^{2}))g
\end{eqnarray}
into our  account, \eqref{a1} and \eqref{al}  gives the desired inequality.
\end{proof}

\begin{proof}[Proof of Theorem \ref{mt}]  
By Duhamel's formula, we note that \eqref{HTEH} can be written in the equivalent form
\begin{equation}\label{df1}
u(\cdot, t)= U(t)u_{0}-i\mathcal{A} (K\ast |u|^2)u
\end{equation}
where
\begin{equation*}
 U(t)=e^{itH} \ \text{and} \  (\mathcal{A}v)(t,x)=\int_{0}^{t}U(t- \tau)\, v(t,x) \, d\tau.
\end{equation*}
We first  prove the local existence on $[0,T)$ for some $T>0.$  
We consider now the mapping
\begin{equation*}
\mathcal{J}(u)= U(t)u_{0}-i\int_{0}^{t}U(t-\tau) \, [(K\ast |u|^{2}(\tau))u(\tau)] \, d\tau.
\end{equation*}
By Minkowski's inequality for integrals,  Theorem \ref{mso}, and  Lemma \ref{iml}, we obtain
\begin{eqnarray*}
\left\| \int_{0}^{t} U(t-\tau) [(K\ast |u|^{2}(\tau)) u(\tau)]  \, d\tau \right\|_{M^{p,p}} 
   &\leq & T  \, \|(K\ast |u|^{2}(t)) u(t)\|_{M^{p,p}} \nonumber \\
   & \leq & T\|u(t)\|_{M^{p,p}}^{3}.
\end{eqnarray*}
By  Theorem \ref{mso}, and using above inequality, we have
\begin{eqnarray*}
\|\mathcal{J}u\|_{C([0, T], M^{p,p})} \leq   \|u_{0}\|_{M^{p,p}} + c T \|u\|_{M^{p,p}}^{3},
\end{eqnarray*}
for some universal constant $c.$
 For $M>0$, put  $$B_{T, M}= \{u\in C([0, T], M^{p,p}(\mathbb R^{d})):\|u\|_{C([0, T], M^{p,p})}\leq M \},$$  which is the  closed ball  of radius $M$, and centered at the origin in  $C([0, T], M^{p,q}(\mathbb R^{d}))$.  
Next, we show that the mapping $\mathcal{J}$ takes $B_{T, M}$ into itself for suitable choice of  $M$ and small $T>0$. Indeed, if we let, $M= 2\|u_{0}\|_{M^{p,p}}$ and $u\in B_{T, M},$ it follows that 
\begin{eqnarray*}
\|\mathcal{J}u\|_{C([0, T], M^{p,p})} \leq  \frac{M}{2} + cT M^{3}.
\end{eqnarray*}
We choose a  $T$  such that  $c TM^{2} \leq 1/2,$ that is, $T \leq \tilde{T}(\|u_0\|_{M^{p,q}}, d, \gamma)$ and as a consequence  we have
\begin{eqnarray*}
\|\mathcal{J}u\|_{C([0, T], M^{p,p})} \leq \frac{M}{2} + \frac{M}{2}=M,
\end{eqnarray*}
that is, $\mathcal{J}u \in B_{T, M}.$
By Lemma \ref{cl}, and the arguments as before, we obtain
\begin{eqnarray*}
\|\mathcal{J}u- \mathcal{J}v\|_{C([0, T], M^{p,p})} \leq \frac{1}{2} \|u-v\|_{C([0, T], M^{p,p})}.
\end{eqnarray*}
Therefore, using the  Banach's contraction mapping principle, we conclude that $\mathcal{J}$ has a fixed point in $B_{T, M}$ which is a solution of \eqref{df1}. 

Taking  Proposition \ref{miD} into account, 
to prove Theorem \ref{mt}, it suffices to prove  that  the modulation space  norm  $\|u\|_{M^{p,p}}$ cannot become unbounded in finite time.
In view of \eqref{dc} and to use the Hausdorff-Young inequality we let $1< \frac{d}{d-\gamma} <q \leq 2,$ 
and we obtain
\begin{eqnarray}
\|u(t)\|_{M^{p,p}} & \lesssim &   \|u_{0}\|_{M^{p,p}} + \int_{0}^{t} \|(K\ast |u(\tau)|^{2}) u(\tau)\|_{M^{p,p}}d\tau \nonumber \\
& \lesssim &    \|u_{0}\|_ {M^{p,p}} + \int_{0}^{t} \|K\ast |u(\tau)|^{2}\|_{\mathcal{F}L^{1}} \|u(\tau)\|_{M^{p,p}} d\tau  \nonumber\\
& \lesssim &   \|u_{0}\|_{M^{p,p}} + \int_{0}^{t} \left( \|k_{1}\|_{L^{1}} \|u(\tau)\|_{L^{2}}^{2}+ \|k_{2}\|_{L^{q}} \|\widehat{|u(\tau)|^{2}}\|_{L^{q'}}
\right) \nonumber \\
&&\|u(\tau)\|_{M^{p,p}} d\tau \nonumber\\
& \lesssim & \|u_{0}\|_{M^{p,p}} +  \int_{0}^{t} \left(  \|u_{0}\|_{L^{2}}^{2}+ \||u(\tau) |^{2}\|_{L^{q}}\right)\|u(\tau)\|_{M^{p,p}} d\tau\nonumber\\
& \lesssim & \|u_{0}\|_{M^{p,p}}+  \int_{0}^{t}\|u(\tau)\|_{M^{p,p}} d\tau  +  \int_{0}^{t} \|u(\tau)\|_{L^{2q}}^{2} \|u(\tau)\|_{M^{p,p}} d\tau,\nonumber
 \end{eqnarray}
 where we have used  Theorem \ref{pl},  H\"older's inequality, and  the conservation of the $L^{2}-$norm of $u$. \\
We note that the requirement on $q$ can be fulfilled if and only if $0<\gamma <d/2.$ To apply Proposition \ref{mi}, we let $\beta>1$ and $(2\beta, 2 q)$ is  admissible, that is, $\frac{2}{2\beta}= d \left(\frac{1}{2}- \frac{1}{2q} \right)$ such that $\frac{1}{\beta}= \frac{d}{2} \left( 1 - \frac{1}{q} \right)<1.$ This is possible provided $\frac{q-1}{q} < \frac{2}{d}:$ this condition is compatible with the requirement $q> \frac{d}{d-\gamma}$ if and only if $\gamma < 2.$
 Using the H\"older's inequality for the last integral, we obtain
\begin{eqnarray*}
\|u(t)\|_{M^{p,p}} &  \lesssim  & \|u_0\|_{M^{p,p}} + \int_{0}^{t} \|u(\tau)\|_{M^{p,p}} d\tau \\
&& + \|u\|_{L^{2\beta}([0, T], L^{2q})}^{2}\|u\|_{L^{\beta'}[0, T], M^{p,p})},
\end{eqnarray*}
where $\beta'$ is the H\"older conjugate exponent of $\beta.$
Put
$$h(t):=\sup_{0 \leq \tau \leq t} \|u(\tau)\|_{M^{p,p}}.$$
For a given $T>0,$ $h$ satisfies an estimate of the form
$$h(t)\lesssim \|u_{0}\|_{M^{p,p}}+ \int_{0}^{t} h(\tau) d\tau + C_0(T) \left( \int_{0}^{t}h(\tau)^{\beta'} d\tau \right)^{\frac{1}{\beta'}},$$
provided that $0 \leq t \leq T,$ and where we have used the fact that $\beta'$ is finite.
Using the H\"older's inequality we infer that
$$h(t)\lesssim   \|u_{0}\|_{M^{p,p}} + C_{1}(T) \left(\int_{0}^{t} h(\tau)^{\beta'}d \tau \right)^{\frac{1}{\beta'}}.$$
Raising the above estimate to the power $\beta'$, we find that
$$h(t)^{\beta'} \lesssim  C_{2}(T) \left( 1+\int_{0}^{t} h(\tau)^{\beta'} d\tau\right).$$ 
In view of  Gronwall inequality, one may conclude that  $h\in L^{\infty}([0, T]).$ Since $T>0$ is arbitrary, $h\in L^{\infty}_{loc}(\mathbb R),$ and  the proof of Theorem \ref{mt}  follows.
\end{proof}
\subsection{Global well-posedness in $M^{p,p}$ for the potential in $M^{\infty,1}(\R^d)$.}\label{pmtv}
In this section we prove Theorem \ref{mtv}.
\begin{Lemma}\label{lcg} Let  $K\in M^{\infty,1}(\mathbb R^d),$  and $1\leq p, q \leq 2.$ For $f \in M^{p,q}(\R^d),$ we have $$ \|(K\ast |f|^{2}) f\|_{M^{p,q}} \lesssim\|f\|_{M^{p,q}}^{3} ,$$and
$$\| (K\ast |f|^{2})f - (K\ast |g|^{2})g\|_{M^{p,q}} \lesssim  (\|f\|_{M^{p,q}}^{2}+\|f\|_{M^{p,q}}\|g\|_{M^{p,q}}+ \|g\|_{M^{p,q}}^{2}) \|f-g\|_{M^{p,q}}.$$
\end{Lemma}
\begin{proof}
We note that $L^1\subset M^{1, \infty}$ (see Lemma \ref{rl} \eqref{el}) and using  Theorem \ref{pl}, we have
 \[ \|(K\ast |f|^{2}) f\|_{M^{p,q}} \lesssim  \| K\ast |f|^2\|_{M^{\infty,1}} \|f\|_{M^{p,q}} \lesssim   \||f|^2\|_{M^{1, \infty}} \|f\|_{M^{p,q}} \lesssim   \|f\|_{L^2}^2 \|f\|_{M^{p,q}} \lesssim \|f\|_{M^{p,q}}^3.\] This proves the first inequality.  In view of \eqref{inpc}, Theorem \ref{pl}, and exploiting ideas from the first inequality gives the second inequality. 
\end{proof}
\begin{proof}[Proof of Theorem \ref{mtv}]   We note that \eqref{HTEH} can be written in the equivalent form
\begin{equation}\label{df11}
u(\cdot, t)= U(t)u_{0}-i\mathcal{A} (K\ast |u|^2)u
\end{equation}
where
\begin{equation*}
  U(t)=e^{itH} \ \text{and}  \  (\mathcal{A}v)(t,x)=\int_{0}^{t}U(t- \tau)\, v(t,x) \, d\tau.
\end{equation*}
We first  prove the local existence on $[0,T)$ for some $T>0.$  
We consider now the mapping
\begin{equation*}
\mathcal{J}(u)= U(t)u_{0}-i\int_{0}^{t}U(t-\tau) \, [(K\ast |u|^{2}(\tau))u(\tau)] \, d\tau.
\end{equation*}
\noindent
By Minkowski's inequality for integrals,  Theorem \ref{mso}, and  Lemma \ref{lcg}, we obtain
\begin{eqnarray*}
\left\| \int_{0}^{t} U(t-\tau) [(K\ast |u|^{2}(\tau)) u(\tau)]  \, d\tau \right\|_{M^{p,p}} 
   & \leq &  cT \|u(t)\|_{M^{p,p}}^{3},
\end{eqnarray*}
 for some universal constant $c.$
By Theorem \ref{mso} and the above inequality, we have
\begin{eqnarray*}
\|\mathcal{J}u\|_{C([0, T], M^{p,p})} \leq  \|u_{0}\|_{M^{p,p}} + cT  \|u\|_{M^{p,p}}^{3}.
\end{eqnarray*}
For $M>0$, put  $$B_{T, M}= \{u\in C([0, T], M^{p,p}(\mathbb R^{d})):\|u\|_{C([0, T], M^{p,p})}\leq M \},$$  which is the  closed ball  of radius $M$, and centered at the origin in  $C([0, T], M^{p,q}(\mathbb R^{d}))$.  
Next, we show that the mapping $\mathcal{J}$ takes $B_{T, M}$ into itself for suitable choice of  $M$ and small $T>0$. Indeed, if we let, $M= 2\|u_{0}\|_{M^{p,p}}$ and $u\in B_{T, M},$ it follows that 
\begin{eqnarray*}
\|\mathcal{J}u\|_{C([0, T], M^{p,p})} \leq  \frac{M}{2} + cT M^{3}.
\end{eqnarray*}
We choose a  $T$  such that  $c TM^{2} \leq 1/2,$ that is, $T \leq \tilde{T}(\|u_0\|_{M^{p,p}})$ and as a consequence  we have
\begin{eqnarray*}
\|\mathcal{J}u\|_{C([0, T], M^{p,p})} \leq \frac{M}{2} + \frac{M}{2}=M,
\end{eqnarray*}
that is, $\mathcal{J}u \in B_{T, M}.$
By Lemma \ref{lcg}, and the arguments as before, we obtain
\begin{eqnarray*}
\|\mathcal{J}u- \mathcal{J}v\|_{C([0, T], M^{p,p})} \leq \frac{1}{2} \|u-v\|_{C([0, T], M^{p,p})}.
\end{eqnarray*}
Therefore, using Banach's contraction mapping principle, we conclude that $\mathcal{J}$ has a fixed point in $B_{T, M}$ which is a solution of \eqref{HTEH}. 

Indeed, the solution constructed before is global in time: in view of the conservation of $L^{2}$ norm, Theorem \ref{pl}, and Lemma \ref{rl}, we have
\begin{eqnarray*}
\|u((t)\|_{M^{p,p}} & \lesssim  &   \|u_{0} \|_{M^{p,p}} + \int_{0}^{t} \|K\ast |u(\tau)|^{2}\|_{M^{\infty,1}} \|u(\tau)\|_{M^{p,p}} d\tau \nonumber \\
& \lesssim &   \|u_{0}\|_{M^{p,p}} + \int_{0}^{t} \|K\|_{M^{\infty,1}} \||u(t)|^{2}\|_{M^{1, \infty}} \|u(\tau)\|_{M^{p,p}} d\tau \nonumber \\
& \lesssim &  \|u_{0}\|_{M^{p,p}} + \int_{0}^{t}  \||u(t)|^{2}\|_{L^1} \|u(\tau)\|_{M^{p,p}} d\tau \nonumber \\
& \lesssim &   \|u_{0}\|_{M^{p,p}} +  \|u_{0}\|_{L^{2}}^{2} \int_{0}^{t}\|u(\tau)\|_{M^{p,p}} d\tau,
\end{eqnarray*}
and by Gronwall inequality, we conclude that $\|u(t)\|_{M^{p,q}}$ remains bounded on finite time intervals. This completes the proof.
\end{proof}
\subsection{Local well-posedness in $M^{1,1}$ for power type non linearity}.\label{phu} In this subsection we prove Theorem \ref{mtp}.  We start by recalling following  
\begin{Definition}
\label{red}
A complex valued function $F$ defined on  the plane $\mathbb R^{2}$ is said to be real entire, if $F$ has the power series expansion 
\bea\label{pwrexp} F(s, t)= \sum_{m,n=0}^{\infty} a_{mn}\, s^{m} \, t^{n}\eea
that converges absolutely for every $(s,t) \in \R^2. $
\end{Definition}
\noindent
 \textbf{Notations}: \begin{enumerate}
\item For $u: \mathbb R^d \to \C,$ we put $u=u_1+iu_2,$ where $u_1, u_2:\mathbb R^d \to \R$, and write 
\begin{eqnarray*} 
F(u)= F(u_{1}, u_{2}),
\end{eqnarray*}
where  $F:\mathbb R^{2}\to \mathbb C$ is real entire on $\mathbb R^{2}$ with $F(0)=0.$ 
\item  If $F$ is real entire function given by \eqref{pwrexp}, then we denote by  $\tilde{F}$ the function given by the power series expansion
\begin{eqnarray*}
\tilde{F}(s,t) = \sum_{m,n=0}^{\infty} |a_{mn}|\, s^{m} \, t^{n}.
\end{eqnarray*}
Note that $\tilde{F}$ is real entire if $F$ is real entire. Moreover, as a function on $[0,\infty) \times [0,\infty)$, it is monotonically increasing  with respect to each of the variables $s$ and $t$.
\end{enumerate} 
\begin{Proposition}[\cite{dgb-pkr}] \label{Fu-Fv}
Let $F$ be a real entire, $F(0)=0$ and $1\leq p \leq \infty.$ Then
\begin{enumerate}

\item  $\|F(u)\|_{M^{p,1}} \lesssim \tilde{F} (\|u\|_{M^{p,1}}, \|u\|_{M^{p,1}}).$

 \item  For $u, v \in M^{p,1}(\mathbb R^d),$ we have 
\Bea \|F(u_1,u_2) \!\!\!\! &-&\!\!\!\! F(v_1,v_2)\|_{M^{p,1}} \\
&& \lesssim 2 \|u-v\|_{M^{p,1}}   \left[ 
  \left(\widetilde{\partial_x F} +\widetilde{\partial_y F} \right) (\|u\|_{M^{p,1}}+\|v\|_{M^{p,1}}, \|u\|_{M^{p,1}}+\|v\|_{M^{p,1}}) \right].\Eea 
\end{enumerate}
\end{Proposition}

\begin{proof}[Proof of Theorem \ref{mtp}]

Equation  \eqref{HTEH} can be written in the equivalent form
\begin{equation}
u(\cdot, t)= U(t)u_{0}-i\mathcal{A} (K\ast |u|^2)u
\end{equation}
where
\begin{equation*}
  U(t)=e^{itH} \ \text{and}  \  (\mathcal{A}v)(t,x)=\int_{0}^{t}U(t- \tau)\, v(t,x) \, d\tau.
\end{equation*}
 We show that  the mapping 
 \begin{equation*} \label{inteq}
\mathcal{J}(u)= U(t)u_{0}-i\int_{0}^{t}U(t-\tau) \, [F(u(\cdot, \tau)) ] \, d\tau
\end{equation*}
has a unique fixed point in an appropriate functions space, for small $t$. 
For this, we consider the Banach space $X_{T}=C([0, T], M^{1,1}(\R))$, with norm
$$\left\|u\right\|_{X_{T}}=\sup_{t\in [0, T]}\left\|u(\cdot, t)\right\|_{M^{1,1}}, \ (u\in X_{T}).$$ 
Note that if $u \in X_T$, then $u(\cdot, t) \in M^{1,1}(\R^d)$ for each $ t \in[0,  T]$. Now Proposition \ref{Fu-Fv} gives $F(u (\cdot,t)) \in M^{1,1}(\R^d)$ and we have 
\bea  \label{xxtest} \| F(u (\cdot,t)) \|_{M^{1,1}}  &\leq&  \tilde{F}(\|u(\cdot,t)\|_{M^{1,1}} ,\|u(\cdot,t)\|_X )\\ \nonumber
&\leq &  \tilde{F}(\|u\|_{X_T} ,\|u\|_{X_T} ),\eea where the last inequality follows from 
the fact that $\tilde{F}$, is monotonically increasing on $[0,\infty) \times [0,\infty)$ with respect to each of its 
variables. Now an application of Minkowski's inequality for integrals  and the estimate \eqref{xxtest} and Theorem \ref{mso}, yields 
\begin{eqnarray}
\label{ed}
\left\| \int_{0}^{t} U(t-\tau) [F(u(\cdot, \tau))]  \, d\tau \right\|_{M^{1,1}} 
& \leq & \int_{0}^{t} \left\|U(t-\tau) [F(u(\cdot, \tau))] \right\|_{M^{1,1}}  \, d\tau \nonumber\\
   &\leq & T  \, \tilde{F}(\|u\|_{X_T} ,\|u\|_{X_T} )
   \end{eqnarray}
for $0\leq t \leq T$.  Using above  estimates, we see that 
\bea
\label{ee}
\left\|\mathcal{J}(u)\right\|_{X_{T}}&\leq&  \left\|u_{0}\right\|_{M^{1,1}} + T  \tilde{F}(\|u\|_{X_T}, \|u\|_{X_T}) \nonumber \\ 
& \leq &   \left\|u_{0}\right\|_{M^{1,1}} + T \, \|u\|_{X_T} G(\|u\|_{X_T})
\eea
where $G$ is a real analytic function on $[0,\infty)$ such that $\tilde{F}(x,x)= x\, G(x)$.
This factorisation follows  from the fact that the constant term in the power series expansion 
for $\tilde{F}$  is zero, (i.e., $\tilde{F}(0,0)=0)$. We also note that $G$ is increasing on $[0,\infty)$.

 For $M>0$, put  $X_{T, M}= \{u\in X_{T}:\|u\|_{X_{T}}\leq M \}$,  which is the  closed ball  of radius $M$, and centered at the origin in  $X_{T}$.  We claim that 
$$\mathcal{J}:X_{T, M}\to X_{T, M},$$
for suitable choice of  $M$ and small $T>0$. 
Let $C_1\geq 1,$  and putting $M=2 C_{1}\left\|u_{0}\right\|_{M^{1,1}}$, from \eqref{ee} we see that for $u\in X_{T, M}$ and $T\leq 1$
\begin{eqnarray} \label{mapto} 
\left\|\mathcal{J}(u)\right\|_{X_{T}} & \leq & \frac{M}{2}+ T C_1 \,  MG(M) \leq M
\end{eqnarray} 
for  $ T \leq T_1$, where 
\begin{eqnarray} \label{T1} 
T_1= \min \left\{ 1, \frac{1}{2C_1G(M)} \right\}.
\end{eqnarray}
Thus  $\mathcal{J}:X_{T, M}\to X_{T, M},$ for $M=2 C_{1}\left\|u_{0}\right\|_{M^{1,1}}$, and all  $T\leq  T_1$, hence the claim.
  
Now we show that ${\mathcal J}$ satisfies the  contraction estimate 
\begin{eqnarray}\label{pki}
\|\mathcal{J}(u)- \mathcal{J}(v)\|_{X_{T}} \leq \frac{1}{2} \|u-v\|_{X_T}
\end{eqnarray}
on $X_{T,M}$ if $T$ is sufficiently small.
By Theorem \ref{mso},  we see that 
\bea\label{ce}
\|\mathcal{J}(u(\cdot, t)- \mathcal{J}(v(\cdot,t))\|_{M^{1,1}} &\leq& 
\int_{0}^{t}\|U(t-\tau) \, [F(u(\cdot, \tau)) -F(v(\cdot,\tau))]\|_{M^{1,1}}  \, d\tau \nonumber \\
&\leq &  \int_{0}^{t}\|F(u(\cdot, \tau)) -F(v(\cdot,\tau))\|_{M^{1,1}}  \, d\tau,
\eea
 By Proposition \ref{Fu-Fv} this is atmost $$2  \int_0^t  \|u-v\|_{M^{1,1}}   \left[ 
  \left(\widetilde{\partial_x F} +\widetilde{\partial_y F} \right) (\|u\|_{M^{1,1}}+\|v\|_{M^{1,1}}, \|u\|_{M^{1,1}}+\|v\|_{M^{1,1}}) \right] d \tau.$$ Now taking supremum over all $t \in [0,T]$, we see that 
 \bea\label{ju-jv}
 \|\mathcal{J}(u) &-& \mathcal{J}(v)\|_{X_T} \nonumber \\
&\leq& 2T \|u-v\|_{X_T} \left( \widetilde{\partial_x F} +\widetilde{\partial_y F} \right) (\|u\|_{X_T}+\|v\|_{X_T}, \|u\|_{X_T}+\|v\|_{X_T}).\eea
Now if $u$ and $v$ are in $X_{T,M}$, the RHS of \eqref {ju-jv} is at most 
\bea\label{contra28}  2T \|u-v\|_{X_T}  \left( \widetilde{\partial_x F} +\widetilde{\partial_y F} \right) (2M, 2M)\leq  
\frac{ \|u-v\|_{X_T} }{2}\eea
for all $T\leq T_2$, where 
\bea\label{T2}
T_2= \min \left\{1, \left[4 C_1 \left( \widetilde{\partial_x F} +\widetilde{\partial_y F} \right) (2M, 2M)\right]^{-1} \right\}.
\eea

Thus from \eqref{contra28}, we see that the estimate \eqref{pki} holds for all $T<T_2$.  
Now choosing $T^1 = \min \{ T_1,T_2\}$ where $T_1$ is given by  \eqref{T1}, so that
both the inequalities \eqref{mapto} and \eqref{pki} are valid for $T<T^1$. Hence for such
a choice of $T$, ${\mathcal J} $ is a contraction on the Banach space $X_{T,M}$ and hence has a unique fixed point in $X_{T,M}$, by the  Banach's contraction mapping principle. Thus  we conclude that $\mathcal{J}$ has a unique fixed point in $X_{T, M}$ which is a solution of \eqref{HTEH} on $[0,T]$ for any $T<T^1$. Note that 
$T^1$ depends on $\|u_0\|_{M^{1,1}}$.

The arguments above also give the solution for the initial data corresponding to any given time $t_0$, on an interval $[t_0, t_0 +T^1]$ where $T^1$ is given by the same formula with  $\|u(0)\|_{M^{1,1}}$ replaced by $\|u(t_0)\|_{M^{1,1}}$. In other words, the dependence of  the length of the interval of existence on the initial time $t_0$, is only through the norm $\|  u(t_0)\|_{M^{1,1}}$. Thus if the solution exists on $[0,T']$ and if $\|u(T')\|_{M^{1,1}}<\infty$, the above arguments can be carried out again for the initial value problem with the new initial data $u(T')$ to extend the solution  to the larger interval $[0, T'']$. This procedure can be continued and hence we get a solution on maximal interval $[0,T^*).$ This completes the proof.  
\end{proof}
\subsection{Local well-posedness in $M^{p,q}$ with potentials in  $\mathcal{F}L^q$ and   $M^{1,\infty}$}\label{lwhp}
In this section, we prove Theorem  \ref{wht1}.   We start by following elementary 
\begin{Proposition}\label{lsd} If $0<p<q<r \leq \infty,$ then $L^{q}(\R^d) \subset L^p(\R^d)+L^r(\R^d),$ that is, each $f\in L^q(\R^d)$ is the sum of function in $L^p(\R^d)$ and a function in $L^r(\R^d).$
\end{Proposition}
\begin{proof}
If $f\in L^q(\R^d),$ let $E= \{ x:|f(x)|>1 \}$ and set $g= f\chi_{E}$ and $h= f\chi_{E^{c}}.$ Then $|g|^p= |f|^p \chi_{E} \leq |f|^q \chi_{E},$ so $g\in L^p(\R^d),$ and $|h|^r= |f|^{r}\chi_{E^c} \leq |f|^q \chi_{E^c},$ so  $h\in L^r(\R^d).$
\end{proof}

\begin{Lemma}\label{iml} Let  $K\in \mathcal{F}L^q(\mathbb R^d),$  $k\in \mathbb N,$ and $1\leq p \leq 2.$  Then 
\begin{enumerate}
\item  \label{gn1} $ \|(K\ast |f|^{2k}) f\|_{M^{1,1}} \lesssim\|f\|_{M^{1,1}}^{2k+1} \ \  \text{for} \ f\in M^{1,1}(\mathbb R^d)  \   \text{and} \  \ 1< q < \infty.$

\item \label{gn3} $ \|(K\ast |f|^{2}) f\|_{M^{p,\frac{2r}{2r-1}}} \lesssim\|f\|_{M^{p,\frac{2r}{2r-1}}}^{3} \ \  \text{for} \ f\in M^{p,1}(\mathbb R^d) \   \text{and} \  \ 1< q < 2,  q< r.$

\item \label{gn5}  $ \|(K\ast |f|^{2k}) f\|_{M^{p,1}} \lesssim\|f\|_{M^{p,1}}^{2k+1} \ \  \text{for} \ f\in M^{p,1}(\mathbb R^d)  \ (1\leq  p \leq  \infty)$ and $K\in M^{1, \infty}(\mathbb R^d) \supset L^{1}(\mathbb R^d)$  $ (L^{1}(\mathbb R^d) \subset \mathcal{F}L^1(\mathbb R^d) )$.

\end{enumerate}
\end{Lemma}
\begin{proof} 
Let $K \in  \mathcal{F} L^q(\mathbb R^d) (1< q \leq  \infty).$ Then by  Proposition \ref{lsd}, we get $k_1 \in L^1(\mathbb R^d)$ and $k_2 \in L^{\infty}(\R^d)$ so that
\begin{eqnarray}\label{f1}
\widehat{K}=k_1+ k_2.
\end{eqnarray}
By Theorem \ref{pl}, \eqref{f1},  H\"older's inequality, Lemma \ref{rl}\eqref{el}, Lemma \ref{rl}\eqref{fi}, and  Lemma \ref{rl}\eqref{ic}, we obtain
\begin{eqnarray}
\|(K\ast |f|^{2k}) f)\|_{M^{1,1}} & \lesssim  & \|K\ast |f|^{2k}\|_{\mathcal{F}L^{1}} \|f\|_{M^{1,1}}  \lesssim  \left(\|k_{1} \widehat{|f|^{2k}}\|_{L^{1}} + \|k_{2} \widehat{|f|^{2k}}\|_{L^{1}} \right) \|f\|_{M^{1,1}} \nonumber \\
& \lesssim &  \left( \|k_{1}\|_{L^{1}} \| \widehat{ |f|^{2k}}\|_{L^{\infty}}  + \|k_{2}\|_{L^{\infty}} \|\widehat{|f|^{2k}}\|_{L^{1}} \right) \|f\|_{M^{1,1}}\nonumber\\
& \lesssim & \left( \||f|^{2k}\|_{L^{1}} + \|\widehat{|f|^{2k}}\|_{M^{1,1}} \right) \|f\|_{M^{1,1}} \lesssim  \| |f|^{2k}\|_{M^{1,1}} \|f\|_{M^{1,1}} \nonumber \\
& \lesssim &  \|f\|_{M^{1,1}}^{2k+1}.\nonumber
\end{eqnarray}
This completes the proof of statement \eqref{gn1}.  For statement \eqref{gn3}, we note that 
  by Proposition \ref{lsd}, we get $k_1 \in L^1(\mathbb R^d)$ and $k_2 \in L^{r}(\R^d) \ (1<q<r \leq 2)$ so that
  \begin{eqnarray}\label{gn2d}
  \widehat{K}=k_1+ k_2.
  \end{eqnarray}
   Let $\frac{1}{r}+ \frac{1}{r'}=1$ and so  $\frac{1}{r'}+1= \frac{1}{r_1}+ \frac{1}{r_2},$ where $r_1=r_2:= \frac{2r}{2r-1} \in [1,2]$, and $r_1' \in [2, \infty]$ where $\frac{1}{r_1}+ \frac{1}{r_1^{'}}=1.$ 
Now  using  Young's inequality for convolution, Lemma \ref{rl} \eqref{rcs}, Lemma \ref{rl} \eqref{ir}, and Lemma \ref{rl} \eqref{ic},  we obtain
\begin{eqnarray}
\|k_2\widehat{|f|^2} \|_{L^1} & \leq & \|k_2\|_{L^r}  \|\widehat{|f|^2}\|_{L^{r'}} \nonumber\\
& \lesssim &  \|\hat{f} \ast \hat{\bar{f}}\|_{L^{r'}}\nonumber\\
& \lesssim & \|\hat{f}\|_{L^{r_1}} \|\hat{\bar{f}}\|_{L^{r_1}} \nonumber \\
& \lesssim &  \|f\|^2_{M^{\min\{r_1', 2\}, r_1}}  \lesssim \|f\|^2_{M^{2, r_1}}\nonumber
\end{eqnarray}
Since  $f:[\frac{d}{d-\gamma}, \infty]\to \mathbb R, f(r) = \frac{2r}{2r-1}$ is a decreasing function, by  Lemma \ref{rl}\eqref{ir}, we have
\begin{eqnarray}\label{gnd33}
\|k_2 \widehat{|f|^2}\|_{L^1}\lesssim  \|f\|^2_{M^{2, r_1}} \lesssim \|f\|^2_{M^{p,q}}. 
\end{eqnarray}
Note that  $\|k_1 \widehat{|f|^2}\|_{L^1} 
 \lesssim  \|f\|_{M^{p,q}}^2$ (see \eqref{md2}). 
 For statement \eqref{gn5},  we note that by Theorem \ref{pl}, we  have 
 \[ \|(K\ast |f|^{2k}) f\|_{M^{p,1}} \lesssim  \||f|^{2k}\|_{M^{\infty, 1}} \|f\|_{M^{p,1}} \lesssim \|f\|_{M^{p,1}}^{2k+1}.\]
\end{proof}
\begin{Lemma}\label{imlc} The following statements are true.
\begin{enumerate}
\item  For $ K\in \mathcal{F}L^q(\mathbb R^d) \ (1 <q < \infty), k \in \mathbb N$ and  $f\in M^{1,1}(\R^d),$ we have 
\begin{multline*} 
\| (K\ast |f|^{2k}) f- (K\ast |g|^{2k}) g\|_{M^{1,1}} \lesssim  \\
   (\|f\|_{M^{1,1}}^{2k}+\|f\|_{M^{1,1}}^{2k-1}\|g\|_{M^{1,1}} + \cdots + \|g\|_{M^{1,1}}^{2k-1}\|f\|_{M^{1,1}} +  \|g\|_{M^{1,1}}^{2k} ) \|f-g\|_{M^{1,1}}.
\end{multline*}

\item  For $ K\in \mathcal{F}L^q(\mathbb R^d) \ (1 <q < r\leq 2),$ and  $f\in M^{p,\frac{2r}{2r-1}}(\R^d),$ we have 
\begin{multline*} 
\| (K\ast |f|^{2}) f- (K\ast |g|^{2}) g\|_{M^{p,\frac{2r}{2r-1}}} \lesssim  \\
   (\|f\|_{M^{p,\frac{2r}{2r-1}}}^{2}+\|f\|_{M^{p,\frac{2r}{2r-1}}}\|g\|_{M^{p,\frac{2r}{2r-1}}}+ \|g\|_{M^{p,\frac{2r}{2r-1}}}^{2}) \|f-g\|_{M^{p,\frac{2r}{2r-1}}}.
   \end{multline*}

\item  For $K \in M^{1, \infty}(\R^d), k \in \mathbb N$ and  $f\in M^{p,1}(\R^d) \ (1\leq p \leq \infty),$ we have\begin{multline*} 
\| (K\ast |f|^{2k}) f- (K\ast |g|^{2k}) g\|_{M^{p,1}}  \lesssim  \\
   (\|f\|_{M^{p,1}}^{2k}+\|f\|_{M^{p,1}}^{2k-1}\|g\|_{M^{p,1}} + \cdots + \|g\|_{M^{1,1}}^{2k-1}\|f\|_{M^{p,1}} +  \|g\|_{M^{p,1}}^{2k} ) \|f-g\|_{M^{p,1}}.
\end{multline*}
\end{enumerate}
\end{Lemma}
\begin{proof} We notice the identities
\begin{eqnarray*}
(K\ast |f|^{2k})f- (K\ast |g|^{2k})g= (K\ast |f|^{2})(f-g) + (K \ast (|f|^{2k}- |g|^{2k}))g
\end{eqnarray*}
and
\[ x^k- y^k = (x-y) \sum_{n=0}^{k-1} x^{k-1-n}y^n \  \ (x, y\geq 0, k \in \mathbb N).\]

Now exploiting ideas from Lemma \ref{iml}, Lemma \ref{cl},  and in view of  the above identities,  Lemma \ref{rl}, and Proposition \ref{pl},  the proofs can be produced. We omit the details. 
\end{proof}

\begin{proof}[Proof of Theorem \ref{wht1}] In view of Theorem \ref{mso}, Lemmas \ref{iml} and \ref{imlc}, the Banach contraction principle gives the desired result. The details are omitted.
\end{proof}

\noindent
{\textbf{Acknowledgment}:} DGB is very grateful to Professor Sundaram Thangavelu for introducing to him to Hermite multiplier on modulation spaces.   DGB is very grateful to Professor  Kasso Okoudjou for hosting and  arranging research facilities at the University of Maryland. DGB is very grateful to Professor R\'emi Carles  for  his suggestions and for pointed out refernces  \cite{pc, ec}.  DGB is  thankful to  SERB Indo-US Postdoctoral Fellowship (2017/142-Divyang G Bhimani) for the financial support. DGB is also thankful to DST-INSPIRE and TIFR CAM for the  academic leave.

\end{document}